\documentclass[12pt]{amsart}
\usepackage{amsmath,amscd,amsthm,amsfonts, amssymb,amsxtra, mathrsfs, amsrefs}
\usepackage[us,12hr]{datetime}
\usepackage[all]{xy} \SelectTips{cm}{}
\usepackage{hyperref}
  \hypersetup{colorlinks=true,citecolor=blue}
   \usepackage{graphicx}
   \usepackage[font=scriptsize]{caption}

\CDat
\newtheorem{thm}{Theorem}[section]  
\newtheorem*{un-no-thm}{Theorem}
\newtheorem{cor}[thm]{Corollary}     

\theoremstyle{definition}
\newtheorem{defn}[thm]{Definition}   

\theoremstyle{definition}

\theoremstyle{definition}
\theoremstyle{remark}
\newtheorem{rem}[thm]{Remark}

\newtheorem*{intro-rem}{Remark}
\newtheorem*{intro-rems}{Remarks}

\newtheorem{ex}[thm]{Example}

\begin{document}
\title{Probability measures on graph trajectories}
\author[M.~J.~Catanzaro]{Michael J.\ Catanzaro}
\address{Department of Mathematics, Iowa State University, Ames, IA 50011}
\email{mjcatanz@iastate.edu}
\author[V.~Y.~Chernyak] {Vladimir Y.\ Chernyak}
\address{Department of Chemistry, Wayne State University, Detroit, MI 48202}
\email{chernyak@chem.wayne.edu}
\author[J.~R.~Klein]{John R.\ Klein}
\address{Department of Mathematics, Wayne State University, Detroit, MI 48202}
\email{klein@math.wayne.edu}
\begin{abstract} 
The aim of this note is to construct a
probability measure on
the space of trajectories in a continuous time Markov chain 
having a finite state diagram, or more generally which admits a  global bound on
its degree and rates.
Our approach is elementary. 
Our main intention is to fill a gap in the literature and to give some additional details
in the proof of \cite[prop.~4.8]{CCK-fluctuation}.
\end{abstract}
\thanks{}
\maketitle
\setlength{\parindent}{15pt}
\setlength{\parskip}{1pt plus 0pt minus 1pt}
\def\smsh{\wedge}
\def\flush{\flushpar}
\def\dbslash{/\!\! /}
\def\:{\colon\!}
\def\Bbb{\mathbb}
\def\bold{\mathbf}
\def\cal{\mathcal}
\def\End{\text{\rm End}}
\def\Aut{\text{\rm Aut}}
\def\map{\text{\rm map}}
\def\sh{\text{\rm sh}}
\def\orb{\cal O}
\def\hoP{\text{\rm ho}P}
\def\Ch{\text{\rm Ch}}
\def\sperf{\text{\rm sperf}}
\def\perf{\text{\rm perf}}
\def\proj{\text{\rm proj}}
\def\fd{\text{\rm fd}}
\def\cfd{\text{\rm cfd}}
\def\gl{\text{\rm GL}}
\def\sm{\text{\rm sm}}

\setcounter{tocdepth}{1}
\tableofcontents
\addcontentsline{file}{sec_unit}{entry}

\section{Introduction} \label{sec:intro}

As is well known, Markov chains model random walks on graphs. Let $\Gamma$ be a directed graph. Its set of vertices
$\Gamma_0$ represent the states of the system and its edges $\Gamma_1$ indicate transitions between states.
There are two flavors of random walk: those in discrete time and those in continuous time. This note will consider
the continuous time variant.

The dynamics of continuous time random walk are encoded by a master equation
\[
p'(t) = \Bbb H(t)p(t)\, ,
\]
 where $\Bbb H(t)$ is a time dependent matrix of {\it transition rates} and $p(t)$ is a 1-parameter family
of probability distributions on $\Gamma_0$. The solutions to the equation describe the time evolution of probability.
For vertices $i$ and $j$, the matrix entry $\Bbb H(t)_{ij}$ 
is the instantaneous rate of change  at time $t$ in jumping from state $i$ to state $j$ 
along the set of edges of $\Gamma$ having initial vertex $i$ and terminal vertex $j$. The operator
$\Bbb H$ is  called the {\it master operator}; its off diagonal entries are non-negative and the sum
of the entries in any column add to zero. 

Given a continuous time Markov chain with state diagram $\Gamma$, our goal here will be to construct a probability distribution on the space of trajectories in $\Gamma$. By a trajectory in $\Gamma$,
we mean a path of contiguous edges equipped with jump times at each vertex of the path.
 Note that such a probability distribution amounts to a description of the stochastic process associated with the Markov chain. 
 
 \begin{rem}
We apologize to the reader in advance for our somewhat unconventional treatment:  
two of us are algebraic topologists and one is a chemical physicist.
\end{rem}

\section{Preliminaries} 
For a set $T$, let $\binom{T}{2}$ denote the set of its non-empty subsets of cardinality  $2$.
An {\it undirected graph} consists of data
\[
X := (X_0,X_1,\delta)\, ,
\]
in which $X_0$ is the set of vertices, $X_1$ is the set of edges and
\[
\delta\: X_1 \to \tbinom{X_0}{2}
\]
is a function.  
We will always assume that $X$ is locally finite in the sense that the function $\delta$ is a finite-to-one.
With this definition multiple edges connecting a pair of distinct vertices are permitted, but we do not permit loop edges, i.e.,
edges which connect a vertex to itself.

A directed graph $\Gamma$ is defined in a similar way, but where now $\delta$ is replaced by a function
$d\: \Gamma_1 \to \Gamma_0(2)$, where $\Gamma_0(2) := \Gamma_0 \times \Gamma_0\setminus \Delta$, i.e.,
the cartesian product with its diagonal deleted. We write $d = (d_0,d_1)$, where 
$d_i\: \Gamma\to \Gamma_0$ is is the function which assigns to a directed edge its source, respectively target.
Note that the canonical map $\pi\:\Gamma_0(2) \to \binom{\Gamma_0}{2}$
is a double cover, and the composition $\delta:= \pi\circ d$ defines the underlying undirected graph.

\begin{ex} Given an undirected graph $X$, we may construct its {\it double.} This is the directed graph
\[
DX := \Gamma = (\Gamma_0,\Gamma_1,d)\, ,
\]
in which $\Gamma_0 = X_0$ and $\Gamma_1$ is the set of ordered pairs $(i,\alpha) \in  X_0 \times X_1$
in which $i \in \delta(\alpha)$. The function $d\: \Gamma_1 \to \Gamma_0(2)$ is given by
$d(i,\alpha) = (i,j)$, where $\delta(\alpha) = \{i,j\}$.  
\end{ex}

\begin{rem} 
Let $\cal G$ be the category of undirected graphs. An object is an undirected graph and a morphism
$f\: (G_0,G_1,\delta)\to (H_0,H_1,\delta')$ consists of functions $f_i\:G_i \to H_i$, $i=0,1$ such that 
$\delta'f_1(\alpha) = f_0(\delta\alpha)$. Similarly, one has the category $\cal G^+$ of directed graphs.
Then we have an adjoint functor pair
\[
U\: \cal G^+ \leftrightarrows \cal G : D
\]
where $U$ is the forgetful functor and $D$ is given by the double.
\end{rem}

\subsection{Markov chains} Let $\Gamma$ be a directed graph.
A \textit{continuous time Markov chain} with state diagram $\Gamma$ 
is an assignment of 
a continuous function 
  \[ k_{\alpha}\: \Bbb R \to [0,\infty)\, ,
 \]
 to each edge $\alpha\in \Gamma_1$. 
The function $k_{\alpha}$  is called the {\it transition rate} of $\alpha$. 
If $d(\alpha) = (i,j)$, then $k_{\alpha}$ is to be interpreted as
  the instantaneous rate of change of probability in jumping 
  from $i$ to $j$ along $\alpha$.

    \begin{rem} The foundational material on Markov chains can be found in the texts of Norris \cite{Norris}
    and  Stroock 
    \cite{STR14}.
    When the transition rates  are constant, the Markov chain is said to be {\it time homogeneous}.
    When the rates are not constant,
   the chain is said to be {\it time inhomogeneous}.\footnote{In contrast with the homogeneous case, the literature on the inhomogeneous case is scant, with the known results  making strong additional assumptions. The only foundational work
   we are of aware of that treats  the time inhomogeneous case is
    Stroock's  text (cf.~\cite[\S5.5.2]{STR14}).}   
  \end{rem}

 \begin{rem}  
The canonical map $\Gamma\to  DU\Gamma$ is an embedding. Given a Markov chain on $\Gamma$, one has a canonical extension
to $DU\Gamma$ by defining the rates to be zero on those edges which aren't in $\Gamma$. The Markovian dynamics 
of the two chains coincide. From this standpoint, there is nothing to lose by assuming that $\Gamma = DX$ for some undirected graph $X$.
\end{rem}
 
 If $\Gamma$ is infinite, we also require the following growth constraints.
 
 \begin{defn}[Rate Bound] For each $t >0$,  there exists a constant $R$, possibly depending on $t$, such that
 \[
 k_\alpha(s) \le R
 \]
 for $0 \le s \le t$ and every $\alpha \in \Gamma_1$.
 \end{defn}
 
 \begin{defn}[Degree Bound] Let $\deg \: \Gamma_0 \to \Bbb N$ be the function which assigns
 to a vertex its degree, i.e., the number of edges meeting it. 
 There is a positive integer $D$ such that 
 \[
 \deg(i) \le D\, , \quad \text{for all } \quad i \in\Gamma_0\, .
 \]
 \end{defn}

Observe that when $\Gamma$ is a finite, both conditions hold automatically.

 \subsection{The master equation} 
Then the rates define a time dependent square matrix $\Bbb H=\Bbb H(t)$, as follows.
For $i\ne j$, set
\[
h_{ij} = \sum_{d(\alpha) = (i,j)} k_{\alpha}\, ,
\]
where the sum is interpreted as zero when $d^{-1}(i,j)$ is the empty set.
Then the matrix entries of $\Bbb H$ are given by 
 \[
 \Bbb H_{ij} = \begin{cases} 
h_{ij}\, ,\qquad & i\ne j \, ;\\
  -\sum_{\ell\ne i} h_{\ell i } \, ,  \quad  \text{ if }  & i=j\, ,
 \end{cases}
 \] 
 where the indices range over $i,j\in \Gamma_0$.
 The  time dependent matrix  $\Bbb H$ is called the {\it master operator}.
  Associated with $\Bbb H$ is a linear, first order 
 ordinary differential equation
 \begin{equation} \label{eqn:kolmogorov}
 p'(t) = \Bbb Hp(t),\qquad 
\end{equation}
in which $p(t)$ is a one parameter family of (probability) distributions on the set of vertices $\Gamma_0$.  
Equation \eqref{eqn:kolmogorov} is called the {\it (forward) Kolmogorov equation} or the {\it master equation} 
\cite[eqn.~5.5.2]{STR14}. 
Its solutions describe the evolution of an initial distribution $p(0)$.

\begin{rem} If $\Gamma = DX$
and the transition rates are constant with value 1, then $\Bbb H$ is the graph Laplacian of $X$  
and  \eqref{eqn:kolmogorov} is a combinatorial version of the heat (diffusion) equation. 
\end{rem}

\begin{rem} The forward Kolmogorov equation is often written in the literature in adjoint form, i.e., as
\[
q'(t) = q(t) \Bbb W\, ,
\] 
where $q(t) = p(t)^*$ and $\Bbb W = \Bbb H^*$ are the transposed matrices. The backward equation (which we will not consider here) is 
\[
q'(t) = \Bbb W q(t)\, .
\]
\end{rem}

 \subsection{Trajectories}
 A {\it path} of length $n$ in $\Gamma$ consists of a sequence of edges
 \[
\alpha_\bullet := (\alpha_1,\dots ,\alpha_n)\, ,
 \]
 such that $d_1(\alpha_k) =d_0(\alpha_{k+1})$ for $1\le k < n$.
We let
\[
i_k(\alpha_\bullet) := i_k
\]
denote the $k$-th vertex of the path, i.e.,  $i_k = d_0(\alpha_k))$ if
$k \le n$ and $i_{n+1} = d_1(\alpha_n)$.

   A {\it trajectory} of length $n$ and duration $t >0$ is a
pair
 \[
 (\alpha_\bullet,t_\bullet)\, ,
 \]
 such that $\alpha_\bullet$ is a path of length $n$ and $t_\bullet = (t_1,\dots,t_n)$
 is a sequence of real numbers satisfying 
 \[
 0 \le t_1 \le \cdots \le t_n \le t\, .
 \]
In what follows, it will be convenient to set
 \[
 t_0 := 0 \quad \text{ and } \quad t_{n+1} =: t\, .
 \]
 \begin{rem}
 For a vertex $i_k = i_k(\alpha_\bullet)$ of the path $\alpha_\bullet)$, the number $t_k$ is  called the {\it jump time} 
 and the number $w_k := t_k -t_{k-1}$ is called {\it wait time}.
\end{rem}

\section{The probability of a trajectory}

Let $(\Gamma,k_\bullet)$ be as in the previous section.
Given a vertex $i \in \Gamma_0$ and an interval $[a,b]$, the {\it escape rate} at $i$ is
\[
u_i(a,b) := \exp\left( -\sum_{d_1(\alpha) = i} \int_{a}^{b} k_\alpha(s)\,  ds\right)  = 
\exp\left(\int_{a}^{b} h_{ii}(s)\,  ds\right) \, .
 \] 
 Fix an initial probability
 distribution $q\: \Gamma_0 \to \Bbb R_+$.
For $j\in \Gamma_0$, set $q_j := q(j)$.

Let 
 \[
 \cal T(\Gamma,n,t)
 \]
 denote the set of trajectories of $\Gamma$ having length $n$.
Define a function
\[
f\: \cal T(\Gamma,n,t)\to \Bbb R_+
\] 
by the formula
\begin{align*}
 f(\alpha_\bullet,t_\bullet) &=  q_{i_{1}}u_{i_{1}}(0,t_{1}) k_{\alpha_1}(t_1)   
 \cdots u_{i_n}(t_{n-1},t_n) k_{\alpha_n}(t_n)   u_{i_{n+1}}(t_n,t) 
   \, , \\
 &= q_{i_{1}} \prod_{m=1}^{n+1} u_{i_m}(t_{m-1},t_{m}) \prod_{m=1}^{n} k_{\alpha_m}(t_{m})
 \end{align*} 
(compare \cite[eqn.~1.112]{Bellac} in the constant rate case).\footnote{The function $f$ is a discrete analogue
of the Onsager-Machlup Lagrangian \cite{Onsager-Machlup}.}

Consider the master equation
\[
p'(t) = \Bbb H p(t), \quad p(0) = q\, .
 \]
Let 
\[
\cal P(\Gamma,n)
\] denote the set of paths of length $n$ and let $\cal P^i(\Gamma,n) \subset \cal P(\Gamma,n)$
denote the subset of those paths which have terminus $i \in \Gamma_0$.

\begin{thm}\label{thm:distribution} The formal solution to the master equation is the vector valued function 
$p(t)$ whose component at $i\in \Gamma_0$ is given by the expression 
\[
p_i(t) = \sum_{n=0}^\infty\sum_{\scriptscriptstyle \alpha_\bullet\in  \cal P^i(\Gamma,n)} \int_0^t \! \!\! \int_0^{t_n}\! \!\! \cdots \! \!\!\int_0^{t_2} 
f(\alpha_\bullet,t_\bullet) \, dt_1 \cdots dt_n\, .
\]
\end{thm}

\begin{proof} Write $\Bbb H = A_0 + A_1$, where $A_0$ is the diagonal matrix with entries
$h_{ii}$.  For $\epsilon >0$, set $\Bbb H_\epsilon = A_0 + \epsilon A_1$. Consider the equation
\begin{equation} \label{eqn:master-epsilon}
\dot p = H_\epsilon p,\quad p(0) = q\, .
\end{equation}
We seek a formal solution $p = p^0 + \epsilon p^1 + \epsilon^2 p^2 + \cdots$ with $p^0(0) = q$ and $p^n(0) = 0$ for $n >0$. 
Once such a solution is found, we set $\epsilon = 1$ to obtain the formal solution to the master equation.

Expanding 
\eqref{eqn:master-epsilon} in $\epsilon$, we obtain the linear system
\begin{equation} \label{eqn:iteration}
\dot p^n = A_0p^n + A_1p^{n-1}\, , \quad n = 0,1,2,\dots
\end{equation}
where by convention $p^{-1} := 0$. 

For $i\in \Gamma_0$, the $i$-th equation of the system is the first order linear differential equation
\begin{equation} \label{eqn:vertex}
\dot p_i^n = h_{ii}p_i^n + \sum_{j\ne i} h_{ij} p_j^{n-1}\, .
\end{equation}
If $n = 0$, the system is uncoupled and separation of variables gives
\[
p_i^0 = q_i e^{\int_0^t h_{ii}(t_1) dt_1} = q_i u_i(0,t)\, .
\]
For $n >0$, the solution to \eqref{eqn:vertex}
can be iteratively solved using the integrating factor.  The first
iteration gives
\begin{align*}
p_{i}^n(t) &= \sum_{j\ne i} \int_0^t e^{\int_{t_n}^t h_{ii}\, dt_{n-1}} h_{ij}(t_n) p_j^{n-1}(t_n) \, dt_n\\
          &= \sum_{\alpha_{n}} \int_0^t u_{i_{n+1}}(t_n,t) k_{\alpha_n}(t_n) p_{i_n} ^{n-1}(t_n) \,  dt_n\, .
\end{align*}
where the second sum is over all edges $\alpha_{n}$ with terminus $i = i_{n+1}$ and whose source is denoted
in the integrand by $i_n$.
We then repeat the procedure using $p_i^{n-1}$ in place of $p_i^{n}$, to obtain
\[
p_i^n(t) = \sum_{\alpha_\bullet} \int_0^t \! \!\! \int_0^{t_n} u_{i_{n+1}}(t_n,t) k_{\alpha_n}(t_n) u_{i_n}(t_{n-1},t_n)
k_{\alpha_{n-1}}(t_{n-1})p_{i_{n-1}}^{n-2}(t_{n-1}) \,  dt_{n-1} dt_{n}\, ,
\]
where the  sum is indexed over all paths of length two $\alpha_\bullet = (\alpha_{n-1},\alpha_n)$ 
satisfying $d(\alpha_{n-1}) = (i_{n-1},i_n)$, $d(\alpha_n) = (i_n,i_{n+1})$, and $i_{n+1} = i$.
Applying this procedure a total of $n$ times results in the desired expression for $p^n_i(t)$.
\end{proof}

Let $\cal T(\Gamma,t)$ denote the space of trajectories of duration $t$ of arbitrary length.

\begin{cor}\label{cor:distribution}  Let $p(t)$ denote the formal solution to the master equation.
 Then $p(t)$ is  a probability distribution on $\Gamma_0$ for every $t\ge 0$. In particular, the function $f$ 
is a probability density on $\cal T(\Gamma,t)$.
\end{cor} 

\begin{proof} As the rate bound holds,
there is a constant $C>0$, independent of $n$, such $f(\alpha_\bullet,t_\bullet) \le C^n$ for all
trajectories of length $n$. As the degree bound holds, there is global bound
$D$ on the degree function, so the number of paths of length $n$ terminating at a vertex $i$ is at most $D^n$.
Consequently,
\[
0\le \sum_{\alpha_\bullet} \int_0^t\! \!\! \int_0^{t_n} \! \!\! \cdots\! \!\! \int_0^{t_2} 
f(\alpha_\bullet,t_\bullet) \, dt_1 \cdots dt_n \le \frac{(CDt)^n}{n!}\, ,
\]
where the sum ranges over paths of length $n$ with terminus $i$. Here, we have used 
the fact that $t^n/n!$ is the volume of the $n$-simplex $0\le t_1\le \cdots t_n \le t$.
By the comparison test, $\sum_n p^n_i(t)$ converges. Therefore $p(t) = \sum_n p^n(t)$ also converges.

Let $\bold 1\: \Gamma_0 \to \Bbb R$ be the row vector which is identically one at every vertex. It will
suffice to show that $\bold 1 \cdot  p(t) = 1$.
Observe that $\bold 1\cdot \Bbb H = 0$, since the entries of $\Bbb H$ in any column add to zero.
Then for all $t$ we have
\[
\frac{d}{dt} (\bold 1 \cdot p(t)) = \bold 1 \cdot p'(t) = \bold 1 \cdot \Bbb Hp(t) = 0\, .
\]
Consequently, $\bold 1 \cdot p(t)$ is a constant. But $\bold 1 \cdot p(0) = 1$, hence $\bold 1 \cdot p(t) = 1$ for all $t$. 
\end{proof}

\section{Fundamental solutions} Consider the master equation 
with initial distribution $p(0) = \delta_{i}$ for a fixed vertex $i$, where $\delta_i(j) = \delta_{ij}$ is the Kronecker delta function.
The solution to this equation is called a {\it fundamental solution} and will be denoted by $u(i,t)$.

In this case the density $f$ is supported on the set of trajectories
$(\alpha_\bullet,t_\bullet)$ with initial vertex $i$. 
Let $\cal T_i(\Gamma,t) \subset \cal T(\Gamma,t)$ denote the subspace of trajectories which start at the vertex $i$.

\begin{cor} With respect to this assumption, the function 
$f$ is a probability density on $\cal T_i(\Gamma,t)$.
\end{cor}

\begin{rem} The general solution $p(t)$ to the master equation with initial condition $p(0) = q$
is obtained from the fundamental solutions using the identity
\begin{equation} \label{eqn:fund}
p(t) = \sum_{j\in \Gamma_0} q_j u(j,t) \, .
\end{equation}
\end{rem}

\begin{defn} Define the {\it propagator} $K\: \Gamma_0 \times \Gamma_0 \times \Bbb R_+\to \Bbb R_+$  by
\[
K(i,j,t) = u_j(i,t)\, ,
\]
i.e., the probability of the set of trajectories of duration $t$ which start at vertex $i$ and terminate at vertex $j$.
Note the initial condition $K(i,j,0) = \delta_{ij}$. 
\end{defn}

Setting $\psi(x,t) := p_x(t)$, equation \eqref{eqn:fund} becomes
\[
\psi(y,t) = \sum_{x\in \Gamma_0} K(x,y,t) \psi(x,0) = \int_{x\in \Gamma_0} K(x,y,t) \psi(x,0)  \, ,
\]
which is familiar to the physics literature. By Theorem \ref{thm:distribution}, we obtain the path integral representation
\[
K(x,y,t) = \sum_{n=0}^\infty\sum_{\scriptscriptstyle \alpha_\bullet\in  \cal P^y_x(\Gamma,n)} \int_0^t\! \!\!\int_0^{t_n} \!\!\!\cdots\!\!\! \int_0^{t_2} 
f(\alpha_\bullet,t_\bullet) \, dt_1 \cdots dt_n\, ,
\]
where $\cal P^y_x(\Gamma,n)$ is the set of paths of length $n$ which start at $x$ and terminate at $y$. Furthermore, the series converges if
$\Gamma$ satisfies the rate and degree bounds.

\begin{ex}
Let $X$ be an $r$-regular graph, i.e., the number of edges meeting each vertex is $r$.
Assume that the rates $k_\bullet$ are constant with value one. Then the master operator $\Bbb H$ is the
negative of the graph Laplacian.
In this instance elementary to check that $f(\alpha_\bullet,t_\bullet) = \delta_{i} e^{-rt}$. Let 
$
\pi_n(i,j)
$
 be the number
of paths of length $n$ in $\Gamma = DX$ from $i$ to $j$. 
 Then a straightforward calculation shows
\[
K(i,j,t) = e^{-rt}\sum_{n=0}^\infty \pi_n(i,j) \frac{t^n}{n!}\, .
\]
In this case, $K$ is the combinatorial {\it heat kernel.}

As $\Gamma$ is $r$-regular, there are precisely $r^n$ paths of length $n$ which start at $i$. Set
\[
\phi_n(i,j) := \frac{\pi_n(i,j)}{r^n}\, .
\] 
Then  $\phi_n(i,j)$ is the probability of the set of paths ({\it not} trajectories),  which end
at  $j$ after $n$-jumps, given that such paths start at $i$ (where the probability of jumping across an edge meeting any vertex is $1/r$).

Let
\[
P(n) = \frac{(rt)^ne^{-rt}}{n!}\, .
\]
Then $P$ is the Poisson probability mass function with parameter $\lambda = rt$.
Consequently,
\[
K(i,j,t) =  \sum_{n=0}^\infty \phi_n(i,j) P(n) =   {\Bbb E}[\Phi_{ij}]\, ,
\]
is the (Poisson) expected value of the random variable $\Phi_{ij}(n):= \phi_n(i,j)$.
Summarizing, the continuous time random walk on an $r$-regular graph with uniform rate $1/r$ 
may be thought of as a discrete time random walk subordinated to a Poisson process (cf.~\cite[chap.~X\S7]{Feller}).
\end{ex}

\bibliography{dummies}



\end{document}